\documentclass[12pt,a4wide]{article}
\usepackage{graphics}
\usepackage{graphicx}
\usepackage{caption}
\usepackage{floatrow}
\usepackage{float}
\usepackage{pifont}
\usepackage{subcaption}
\usepackage{enumerate}
\usepackage{epsfig}
\usepackage[mathscr]{euscript}
\usepackage{mathrsfs}
\usepackage{amssymb}
\usepackage{mathtools}
\usepackage{amsmath,amsthm,amsfonts}

\newtheorem{theorem}{Theorem}[section]

\newtheorem{proposition}[theorem]{Proposition}

\newtheorem{definition}[theorem]{Definition}
\newtheorem{conjecture}[theorem]{Conjecture}
\newtheorem{lemma}[theorem]{Lemma}

\newtheorem{remark}[theorem]{Remark}

\newcommand*{\QEDB}{\hfill\ensuremath{\square}}%

\begin{document}
\title{On the order of regular graphs with fixed second largest eigenvalue}
\author{Jae Young Yang$^1$\footnote{J.Y. Yang is partially supported by the National Natural Science
Foundation of China (No. 11371028).}, Jack H. Koolen$^{2,3}$\footnote{J.H. Koolen is partially supported by the National
Natural Science Foundation of China (No. 11471009 and No. 11671376).} 
\\ \\
\small ${}^1$ School of Mathematical Sciences,\\
\small Anhui University, \\
\small 111 Jiulong Road, Hefei, 230039, Anhui, PR China\\
\small $^2$ School of Mathematical Sciences,\\
\small University of Science and Technology of China, \\
\small 96 Jinzhai Road, Hefei, 230026, Anhui, PR China\\
\small $^3$ Wen-Tsun Wu Key Laboratory of CAS,\\
\small 96 Jinzhai Road, Hefei, 230026, Anhui, PR China\vspace{4pt}\\
\small {\tt e-mail : piez@naver.com, koolen@ustc.edu.cn}\vspace{-3pt}
}
\date{}
\maketitle

\begin{abstract}

Let $v(k, \lambda)$ be the maximum number of vertices of a connected $k$-regular graph with second largest eigenvalue at most $\lambda$. The Alon-Boppana Theorem implies that $v(k, \lambda)$ is finite when $k > \frac{\lambda^2 + 4}{4}$. In this paper, we show that for fixed $\lambda \geq1$, there exists a constant $C(\lambda)$ such that $2k+2 \leq v(k, \lambda) \leq 2k + C(\lambda)$ when $k > \frac{\lambda^2 + 4}{4}$. 

\end{abstract}

\textbf{Keywords} : smallest eigenvalue, Hoffman graph, Alon-Boppana Theorem, co-edge-regular graph

\textbf{AMS classification} : 05C50, 05C75, 05C62
\section{Introduction}

Let $v(k, \lambda)$ be the maximum order of a connected $k$-regular graph with second largest eigenvalue at most $\lambda$. For  $\lambda \geq 2\sqrt{k-1}$, it is known that $v(k, \lambda)$ is infinite from the existence of infinite families of bipartite regular Ramanujan graphs \cite{Ramanujan}. The Alon-Boppana Theorem \cite{Alon, ab1, ab2, ab3, ab4, nilli1, nilli2} states:

\begin{theorem}\label{alon}
For any integer $k\geq 3$ and real number $\lambda < 2\sqrt{k-1}$,  the number $v(k, \lambda)$ is finite.  
\end{theorem}

In this paper, we will look at the behavior of $v(k, \lambda)$ when $\lambda$ is fixed and $k$ goes to infinity. Our main theorem is:

\begin{theorem}\label{main}
Let $\lambda$ be an integer at least $1$. Then there exists a constant $C_1(\lambda)$ such that $2k+2 \leq v(k, \lambda) \leq 2k + C_1(\lambda)$ holds for all $k > \frac{\lambda ^2 + 4}{4}$.
\end{theorem}


 For fixed real number $\lambda \geq 1$, define $T(\lambda)$ as 

$$T(\lambda) := \limsup_{k \rightarrow \infty} v(k,\lambda) -2k  .$$

Because of Theorem \ref{main}, $T(\lambda)$ is well-defined. We will show that $T(\lambda)\geq 2\lambda$ holds for fixed positive integer $\lambda$.

The proof of Theorem \ref{main} is based on the following proposition. In order to state this proposition, we need to introduce the next notion. For a vertex $x$ of a graph $G$, let $\Gamma_i(x)$ be the set of vertices which are at distance $i$ from $x$.

\begin{proposition}\label{tool}

Let $\lambda$ be a real number at least $1$. Then there exists a constant $M(\lambda) \geq \lambda^3$ such that, if $G$ is a graph satisfying 

\begin{enumerate}[(i)]

\item every pair of vertices at distance $2$ has at least $M(\lambda)$ common neighbors,

\item the smallest eigenvalue of $G$, $\lambda_{\min}(G),$ satisfies  $\lambda_{\min}(G)\geq -\lambda$,

\end{enumerate}

\noindent then $G$ has diameter $2$ and $|\Gamma_2(x) | \leq \lfloor \lambda \rfloor \lfloor\lambda^2 \rfloor$ for all $x \in V(G)$.

\end{proposition}

\begin{remark}

Neumaier \cite{-m} mentioned that Hoffman gave a very large bound on the intersection number $c_2$ of strongly regular graphs. This may imply that Proposition \ref{tool} was already known by Hoffman. However, we could not find it in the literature.

\end{remark}

To prove Proposition \ref{tool}, we use a combinatorial object named Hoffman graphs. The definition and basic properties of Hoffman graphs are given in Section 2. In Section 3, we prove Proposition \ref{tool}. In Section 4, we present some known facts on the number $v(k, \lambda)$, and, in Section 5, we prove Theorem \ref{main} by using Proposition \ref{tool}. In Section 6, we discuss the behavior of the number $T(\lambda)$ for a fixed positive integer $\lambda$. In the last section, we give two more applications of Proposition \ref{tool} for the classes of co-edge regular graphs and amply regular graphs.









\section{Hoffman graphs}

In this section, we introduce the definition and basic properties of Hoffman graphs. Hoffman graphs were defined by Woo and Neumaier \cite{Woo} following an idea of Hoffman \cite{ Hoff1977}. For more details or proofs, see \cite{Jang, KKY, Woo}.

\subsection{Definition and properties of Hoffman graphs}

\begin{definition}

A Hoffman graph $\mathfrak{h}$ is a pair $(H, \ell)$ of a graph $H$ and a labeling map $\ell : V(H) \rightarrow \{{\rm \bf fat,slim}\}$ satisfying two conditions:

\begin{enumerate}[(i)]

\item the vertices with label {\rm \bf fat} are pairwise non-adjacent,

\item every vertex with label {\rm \bf fat} has at least one neighbor with label {\rm \bf slim}.

\end{enumerate}

\end{definition}

The vertices with label {\rm \bf fat} are called {\it fat} vertices, and the set of fat vertices of $\mathfrak{h}$ are denoted by $V_{\rm  fat}(\mathfrak{h})$. The vertices with label {\rm \bf slim} are called {\it slim} vertices, and the set of slim vertices are denoted by $V_{\rm  slim}(\mathfrak{h})$. Now, we give some definitions.

\begin{definition}

For a Hoffman graph $\mathfrak{h}$, a Hoffman graph $\mathfrak{h}_1 = (H_1, \ell_1)$ is called an {\it induced Hoffman subgraph} of $\mathfrak{h}$ if $H_1$ is an induced subgraph of $H$ and $\ell(x) = \ell_1 (x)$ for all vertices $x$ of $H_1$.

\end{definition}

\begin{definition}

Two Hoffman graphs $\mathfrak{h}=(H, \ell)$ and $\mathfrak{h}'=(H', \ell')$ are called {\it isomorphic} if there exists a graph isomorphism $\psi$ from $H$ to $H'$ such that $\ell(x) = \ell'(\psi(x))$ for all vertices $x$ of $H$. 

\end{definition}

\begin{definition}

For a Hoffman graph $\mathfrak{h} = (H, \ell)$, let $A(H)$ be the adjacency matrix of $H$ with a labeling in which the fat vertices come last. Then

$$A(H) = \begin{pmatrix}
A_{\rm  slim} & C \\
C^T & O
\end{pmatrix},$$

\noindent where $A_{\rm  slim}$ is the adjacency matrix of the subgraph of $H$ induced by slim vertices and $O$ is the zero matrix.

The real symmetric matrix $S(\mathfrak{h}) = A_{\rm  slim} - CC^T$ is called the {\it special matrix} of $\mathfrak{h}$, and the eigenvalues of $\mathfrak{h}$ are the eigenvalues of $S(\mathfrak{h})$.

\end{definition}

For a Hoffman graph $\mathfrak{h}$, we focus on its smallest eigenvalue in this paper. Let $\lambda_{\min}(\mathfrak{h})$ denote the smallest eigenvalue of $\mathfrak{h}$. Now, we discuss some spectral properties of $\lambda_{\min}(\mathfrak{h})$ without proofs. 

\begin{lemma}{\rm \cite[Corollary 3.3]{Woo}} If $\mathfrak{h}'$ is an induced Hoffman subgraph of $\mathfrak{h}$, then $\lambda_{\min}(\mathfrak{h}') \geq \lambda_{\min}(\mathfrak{h})$ holds.

\end{lemma}

\begin{theorem}\label{OH}{\rm \cite[Theorem 2.2]{KKY}} Let $\mathfrak{h}$ be a Hoffman graph. For a positive integer $p$, let $G(\mathfrak{h}, p)$ be the graph obtained from $\mathfrak{h}$ by replacing every fat vertex of $\mathfrak{h}$ by a complete graph $K_p$ of $p$ slim vertices, and connecting all vertices of the $K_p$ to all neighbors of the original fat vertex by edges. Then

$$ \lambda_{\min}(G(\mathfrak{h}, p)) \geq \lambda_{\min}(\mathfrak{h}), $$

and

$$ \lim_{p\rightarrow \infty} \lambda_{\min}(G(\mathfrak{h}, p)) = \lambda_{\min}(\mathfrak{h}). $$

\end{theorem}

\subsection{Quasi-clique and associated Hoffman graph}

In this subsection, we introduce two terminologies, {\it quasi-clique} and {\it associated Hoffman graph}. Most of this section is explicitly formulated in \cite{KKY}. Note that the term {\it quasi-clique} in this paper is different from the term quasi-clique in $\cite{Woo}$.

For the rest of this section, let $\widetilde{K}_{2m}$ be the graph consisting of a complete graph $K_{2m}$ and a vertex which is adjacent to exactly $m$ vertices of the $K_{2m}$. For a positive integer $m$ at least $2$, let $G$ be a graph which does not contain $\widetilde{K}_{2m}$  as an induced subgraph. For a positive integer $n$ at least $(m+1)^2$, let $\mathcal{C}(n)$ be the set of maximal cliques of $G$ with at least $n$ vertices. Define the relation $\equiv_n^m$ on $\mathcal{C}(n)$ by $C_1 \equiv_n^m C_2$ if every vertex $x \in C_1$ has at most $m-1$ non-neighbors in $C_2$ and every vertex $y \in C_2$ has at most $m-1$ non-neighbors in $C_1$ for $C_1, C_2 \in \mathcal{C}(n)$.

\begin{lemma}{\rm \cite[Lemma 3.1]{KKY}} Let $m, n$ be two integers at least $2$ such that $n \geq (m+1)^2$. Then the relation $\equiv_n^m$ on $\mathcal{C}(n)$ is an equivalence relation. 

\end{lemma}

For a maximal clique $C \in \mathcal{C}(n)$, let $[C]_n^m$ denote the equivalence class containing $C$. Now, we are ready to define the term {\it quasi-clique}.

\begin{definition}

Let $m, n$ be two integers at least $2$ such that $n \geq (m+1)^2$. For a maximal clique $C \in \mathcal{C}(n)$, we define the quasi-clique $Q[C]_n^m$ with respect to the pair $(m,n)$ of $G$, as the subgraph of $G$ induced on the vertices which have at most $m-1$ non-neighbors in $C$.

\end{definition}

By {\rm \cite[Lemma 3.2]{KKY}} and  {\rm \cite[Lemma 3.3]{KKY}}, the quasi-clique $Q[C]_n^m$ is well-defined for $C \in \mathcal{C}(n)$.

Now we introduce the associated Hoffman graphs. In the next proposition, we present a result which is needed to show Proposition 1.3.

\begin{definition}

Let $m, n$ be two integers at least $2$ such that $n \geq (m+1)^2$. Let $[C_1]_n^m, [C_2]_n^m, \cdots, [C_t]_n^m$ be all the equivalence classes of $G$ under $\equiv_n^m$. The associated Hoffman graph $\mathfrak{g} = \mathfrak{g}(G, m, n)$ is the Hoffman graph with the following properties.

\begin{enumerate}[(i)]

\item $V_{\rm slim}(\mathfrak{g}) = V(G)$, and $V_{\rm fat}(\mathfrak{g}) = \{F_1, \dots, F_t\}$, where $t$ is the number of equivalence classes of $G$ under $\equiv_n^m$,

\item the induced Hoffman subgraph of $\mathfrak{g}$ on $V_{\rm slim}(\mathfrak{g})$ is isomorphic to $G$,

\item the fat vertex $F_i$ is adjacent to all vertices of the quasi-clique $Q[C_i]_n^m$ for $i=1,2,\dots, t$.

\end{enumerate}

\end{definition}

\begin{proposition}\label{asso}{\rm \cite[Proposition 4.1]{KKY}} There exists a positive integer $n = n(m, \phi, \sigma,  p) \geq (m+1)^2$ such that for any integer $q \geq n$, and any Hoffman graph $\mathfrak{h}$ with at most $\phi$ fat vertices and at most $\sigma$ slim vertices, the graph  $G(\mathfrak{h}, p)$ is an induced subgraph of $G$, provided that the graph  $G$ satisfies the following conditions:

\begin{enumerate}[(i)]

\item the graph $G$ does not contain  $\widetilde{K}_{2m}$ as an induced subgraph,

\item the associated Hoffman graph $\mathfrak{g} = \mathfrak{g}(G, m, q)$ contains $\mathfrak{h}$ as an induced Hoffman subgraph.

\end{enumerate}

\end{proposition}

\section{Main tool}

In this section, we prove Proposition \ref{tool}, which is the main tool of this paper. Before we prove Proposition \ref{tool}, we first show two lemmas. 

Let $H$ be a graph. Define $\mathfrak{q}(H)$ the Hoffman graph obtained by attaching one fat vertex to all vertices of $H$. Then $\lambda_{\min}(\mathfrak{q}(H)) = -\lambda_{\max}(\overline{H})$, where $\lambda_{\max}(\overline{H})$ is the maximal eigenvalue of the complement $\overline{H}$ of $H$. The Perron-Frobenius Theorem implies the following lemma.

\begin{lemma}\label{min}

Let $H$ be a graph with an isolated vertex $x$. If $\lambda_{\min}(\mathfrak{q}(H)) \geq -\lambda$ for some real number $\lambda \geq 1$, then $H$ has at most $\lfloor \lambda^2 \rfloor + 1$ vertices.
 
\end{lemma}

\begin{proof}

Let $n$ be the number of vertices of $H$. Since $x$ is an isolated vertex of $H$, $x$ is adjacent to all other vertices of $H$ in the complement $\overline{H}$ of $H$. By the Perron-Frobenius Theorem, we have 

$$  \lambda_{\max}(\overline{H}) \geq \lambda_{\max}(K_{1,n-1}) = \sqrt{n-1}$$.

This shows the lemma.
\end{proof}

\begin{lemma}\label{min2} Let $\lambda$ be a real number at least $1$. Then there exist minimum positive integers $t'(\lambda)$ and $m'(\lambda)$ such that both $\lambda_{\min}(K_{2,t'(\lambda)}) < -\lambda$ and $\lambda_{\min}(\widetilde{K}_{2m'(\lambda)}) < -\lambda$ hold.

\end{lemma}

\begin{proof}
Since $\lambda_{\min}(K_{2,t}) = -\sqrt{2t}$ and $\lambda_{\min}(\widetilde{K}_{2m})$ is the smallest eigenvalue of the matrix

$$ \begin{pmatrix}
m-1 & m & 0 \\
m & m-1 & 1\\
0 & m & 0
\end{pmatrix},$$
it is easily checked that 

$$ \lim_{t \rightarrow \infty} \lambda_{\min}(K_{2,t}) = \lim_{m \rightarrow \infty} \lambda_{\min}(\widetilde{K}_{2m}) = -\infty. $$
This shows the existence of $t'(\lambda)$ and $m'(\lambda)$.
\end{proof}

\vspace{0.2cm}
\noindent{\bf Proof of Proposition \ref{tool}.} 
First, we consider the Hoffman graph $\mathfrak{h}^{(\lfloor \lambda +1 \rfloor)}$ with $\lfloor \lambda +1 \rfloor$ fat vertices adjacent to one slim vertex. Then $\lambda_{\min}(\mathfrak{h}^{(\lfloor \lambda +1 \rfloor)}) = -\lfloor \lambda +1 \rfloor < -\lambda$, so there exists an positive integer $p_{0}$ such that  $\lambda_{\min}(G(\mathfrak{h}^{(\lfloor \lambda +1 \rfloor)}, p_{0})) < -\lambda$ by Theorem \ref{OH}.

Next, let $\{H_1, \cdots, H_r\}$ be the set of pairwise non-isomorphic graphs on $\lfloor \lambda^2  \rfloor + 2$ vertices with an isolated vertex. By Lemma \ref{min}, $\lambda_{\min}(\mathfrak{q}(H_i)) < -\lambda$ holds for all $i = 1, \dots, r$. For each $i=1,\dots, r$, there exists positive integers $p_i$ such that $\lambda_{\min}(G(\mathfrak{q}(H_i), p_i)) < -\lambda$ by Theorem \ref{OH}. Set  $p' = \max p_i$.

For the two integers $t' = t'(\lambda)$ and $m' = m'(\lambda)$ of Lemma \ref{min2}, let $n' = n(m', \lfloor \lambda +1 \rfloor, \lfloor \lambda^2 +2 \rfloor, p')$ where $n(m', \lfloor \lambda +1 \rfloor, \lfloor \lambda^2 +2 \rfloor, p')$ is the integer in Proposition \ref{asso}. This means that the associated Hoffman graph $\mathfrak{g}(G, m', n')$ does not contain any of the Hoffman graphs in the set $\{\mathfrak{h}^{(\lfloor \lambda +1 \rfloor)}\} \cup \{\mathfrak{q}(H_i) \mid i =1, .\ldots, r\}$ as induced subgraphs. This implies that the following two conditions hold:

\begin{enumerate}[(i)]

\item for each vertex $x$ of $\mathfrak{g}(G, m', n')$ and one of its fat neighbor $f_x$, the number of vertices which is adjacent to $f_x$ and non-adjacent to $x$ is at most $\lfloor \lambda^2 \rfloor$,

\item every vertex $x$ of $\mathfrak{g}(G, m', n')$ has at most $\lfloor \lambda \rfloor$ fat neighbors.

\end{enumerate}

Now we want to assume that for any two distinct non-adjacent vertices $x$ and $y$ of $G$, they have a common fat neighbor in $\mathfrak{g}(G, m', n')$. To do so, let $M(\lambda)$ be the number $\max\{R(n', t'), \lfloor \lambda^3 +1\rfloor\}$, where $R(n', t')$ denotes the Ramsey number. Recall that the Ramsey number $R(s,t)$ is the minimal positive integer $n$ such that any graph with order $n$ contains a clique of order $s$ or a coclique of order $t$. The property of Ramsey number implies that for two vertices $x, y$ at distance $2$, their common neighborhood contains a clique of size $n'$ or a coclique of size $t'$. Hence, there exists a fat vertex which is adjacent to both $x$ and $y$ in $\mathfrak{g}(G, m', n')$. From (i) and (ii), we conclude that $|\Gamma_2(x)| \leq \lfloor \lambda \rfloor \lfloor\lambda^2  \rfloor$ for all $x \in V(G)$.

Assume that there exists a vertex $y \in \Gamma_3(x)$ for some $x$. Then there exists a vertex $z$ such that $z \in \Gamma_1(x)$ and $z \in \Gamma_2(y)$. The common neighborhood of $y$ and $z$ have at least size $M(\lambda)$ and is contained in $\Gamma_2(x)$. This is impossible. Hence, $G$ has diameter $2$. \QEDB

\section{Some known facts on the number $v(k, \lambda)$}

In this section we give some known facts on the number $v(k, \lambda)$.
We start from the case $\lambda < 0$. If a connected graph $G$ is not complete, $G$ contains $K_{1,2}$ as an induced subgraph. Then by interlacing, $G$ has second largest eigenvalue at least $0$. It implies that if a graph $G$ has negative second largest eigenvalue, $G$ is a complete graph. Thus, $v(k, \lambda) = k+1$ for $\lambda < 0$ and the unique graph with the equality case is the complete graph $K_{k+1}$. 

For $\lambda = 0$, a regular graph with non-positive second largest eigenvalue is a complete multipartite graph \cite[Corollary 3.5.4]{drg}. Among $k$-regular complete multipartite graphs, we can check that the complete bipartite graph $K_{k,k}$ maximizes the number of vertices. Hence $v(k, 0) = 2k$ and the unique graph with the equality case is the complete multipartite graph $K_{k,k}$.

For $\lambda = 1$, let $G$ be a regular graph with second largest eigenvalue at most $1$. Then the complement of $G$ is a regular graph with smallest eigenvalue at least $-2$. Since such regular graphs are classified in \cite{seidel}, we can find the all values of $v(k, 1)$ \cite[Theorem 3.2]{Nozaki}. Especially, $v(k, 1) = 2k+2$ when $k \geq 11$. The equality case is obtained by the complement of the line graph of $K_{2,k+1}$. Note that $2k+2 \leq v(k, 1) \leq 2k+6$ for all $k$.

For other values of $\lambda > 1$,  Cioab\u{a} et al. \cite{Nozaki} found several values of $v(k, \lambda)$ by using a linear programming method. Using the method of  Cioab\u{a} et al., it can be shown that $v(k, \lambda) \leq (\lambda + 2)k +\lambda^3 +\lambda^2 - \lambda$ if $k$ is large enough. Theorem \ref{main} improves this result significantly.

\section{Proof of the main theorem}
Now, we are ready to prove Theorem \ref{main}.

\vspace{0.2cm}

\noindent {\bf Proof of Theorem \ref{main}.} Let $G$ be a $k$-regular graph with second largest eigenvalue $\lambda$ and $v(k, \lambda)$ vertices. Since $v(k, 1) \geq 2k+2$ and $\lambda \geq 1$, $v(k, \lambda) \geq 2k+2$. We only need to show Theorem \ref{main} for large enough $k$, so we may assume that $k > \lambda(\lambda+1)(\lambda+2)$. Now, we consider the complement $\overline{G}$ of $G$. Then $\overline{G}$ is a $l$-regular graph with smallest eigenvalue $-1-\lambda$ and $v(k, \lambda)$ vertices, where $l = v(k, \lambda) - k - 1 \geq k+1$.

Suppose that $l \geq k + C_1(\lambda)$, where $C_1(\lambda) =  M(\lambda+1) -1$, where $M(\lambda)$ is the constant of Proposition \ref{tool}. Let $x$ be a vertex of $\overline{G}$. Then the set of non-neighbors of $x$ has size $k$ and has at least $M(\lambda+1)$ neighbors in the neighborhood of $x$ since $\overline{G}$ is $l$-regular and $l\ \geq k + C_1(\lambda)$. It implies that the set of non-neighbors of $x$ is exactly $\Gamma_2(x)$. By Proposition \ref{tool}, $G$ has diameter $2$ and $|\Gamma_2(x)| \leq (\lambda +1)(\lambda^2 +2\lambda)$. This contradicts to the assumption $k > (\lambda +1)(\lambda^2 +2\lambda)$. Hence, $\l \leq k + C_1(\lambda) -1$ and $v(k, \lambda) = 1 + k + l \leq 2k + C_1(\lambda)$. \QEDB

\section{The behavior of $T(\lambda)$}
Recall that, for fixed real number $\lambda \geq 1$, $T(\lambda)$ is defined as

$$T(\lambda) := \limsup_{k \rightarrow \infty} v(k,\lambda) -2k  .$$
Now we give a result on $T(\lambda)$.

The complement of the line graph of $K_{2,a+1}$, denoted by $\overline{L(K_{2,a+1})}$, is a $a$-regular graph which has $2a+2$ vertices and spectrum $\{[a]^1, [1]^{a}, [-1]^{a}, [-a]^1\}$. We consider the coclique extension of this graph.

\begin{definition}

For an integer $q > 1$, the $q$-coclique extension $\tilde{G}_q$ of a graph $G$ is the graph obtained from $G$ by replacing each vertex $x \in G$ by a coclique $\tilde{X}$ with $q$  vertices, such that  $\tilde{x} \in \tilde{X}$ and $\tilde{y} \in \tilde{Y}$ are adjacent if and only if $x$ and $y$ are adjacent in $G$.

\end{definition}

If $\tilde{G}_q$ is the $q$-coclique extension of $G$, then $\tilde{G}_q$ has adjacency matrix $A\otimes J_q$, where $J_q$ is the all one matrix of size $q$ and $\otimes$ denotes the Kronecker product. This shows that, if a graph $G$ has spectrum 
$$\{[\lambda_0]^{m_0}, [\lambda_1]^{m_1}, \dots, [\lambda_n]^{m_n}\},$$

\noindent then the $q$-coclique extension $\tilde{G}_q$ of $G$ has spectrum

$$\{[q\lambda_0]^{m_0}, [q\lambda_1]^{m_1}, \dots, [q\lambda_n]^{m_n}, [0]^{(q-1)(m_0+m_1+\dots+m_n)}\}.$$

 Hence the $q$-coclique extension of $\overline{L(K_{2,a+1})}$ is a $qa$-regular graph with order  $2qa + 2q$ and spectrum

$$ \{[qa]^1, [q]^a, [-q]^a, [-qa]^1, [0]^{(q-1)(2a+2)}\}.$$

This implies that the  $\lambda$-coclique extension of $\overline{L(K_{2,a+1})}$ has  second largest eigenvalue $\lambda$, and that $v(k, \lambda) \geq 2k+2\lambda$ when $k$ is a multiple of $\lambda$. Hence we have:

\begin{lemma}\label{T}

Let $\lambda$ be a positive integer. Then $T(\lambda) \geq 2\lambda$.
\end{lemma}

Moreover, we have a conjecture on $T(\lambda)$ as follows:

\begin{conjecture}

Let $\lambda$ be a positive integer. Then $T(\lambda) = 2\lambda$.

\end{conjecture}

For $\lambda = 1$, this conjecture is true as $T(1) = 2$.

\section{Applications}

In this section, we introduce two applications of Proposition \ref{tool}. We first consider co-edge regular graphs with parameters $(v, k, c_2)$, which are $k$-regular graphs with $v$ vertices and the property such that every pair of non-adjacent vertices has exactly $c_2$ common neighbors. By applying Proposition \ref{tool} and Theorem \ref{alon}, we obtain the following theorem.

\begin{theorem}\label{co-edge}
Let $\lambda \geq 1$ be a real number.
Let $G$ be a connected co-edge regular graph with parameters $(v, k, c_2)$.  Then there exists a real number $C_2(\lambda)$ (only depending on $\lambda$) such that, if $G$ has smallest eigenvalue at least $-\lambda $, then $c_2 > C_2(\lambda)$ implies that  $v-k-1 \leq \frac{(\lambda -1)^2}{4}+1$ holds.

\end{theorem}

\begin{proof} Let $\ell = v-k-1$, then $|\Gamma_2(x)| = \ell$ for all $x \in V(G)$ since $G$ has diameter $2$. We can apply Proposition \ref{tool} for $C(\lambda) = M(\lambda)-1$ to obtain that either $c_2 \leq C(\lambda)$ or $\ell = v-k-1 \leq \lfloor \lambda \rfloor \lfloor\lambda^2 \rfloor$ holds. Suppose $\ell > \frac{(\lambda -1)^2}{4}+1$. The complement of $G$ is $\ell$-regular and has second largest eigenvalue at most $\lambda -1$, and therefore has at most $v(\ell, \lambda -1)$ vertices. As  $v(\ell, \lambda -1)$ is a finite number by 
Theorem \ref{alon}, we see that the theorem follows, if we take 
$C_2 (\lambda) =\max \{ \max\{ v(\ell, \lambda -1) -\ell-1 \mid \frac{(\lambda -1)^2}{4}+1<  \ell\leq \lfloor \lambda \rfloor \lfloor\lambda^2 \rfloor\}, C(\lambda)\}$. 
\end{proof}

An edge-regular graph with parameters $(v, k, a_1)$ is a $k$-regular graph with $v$ vertices such that any two adjacent vertices have exactly $a_1$ common neighbors. Note that the complement of a co-edge regular graph is edge-regular. 

\begin{remark}

(i) Let $\ell$ be an integer at least 3 and let $\lambda := 2\sqrt{\ell-1}$. 
Take the infinite  family of the bipartite $\ell$-regular Ramanujan graphs, as constructed in  \cite{Ramanujan}. The graphs in this family are clearly edge-regular with $a_1=0$ and have second largest eigenvalue at most $2\sqrt{\ell-1}$. Let $\Gamma$ be a graph in this family with $v$ vertices. Then the complement of $\Gamma$ is co-edge-regular with parameters $(v, v-\ell-1, v-2\ell)$ and has smallest eigenvalue  at least $-1-2\sqrt{\ell-1}$.

This example shows that the upper bound for  $v-k-1$ in Theorem \ref{co-edge} cannot be improved.\\
(ii) For $\lambda =2$, we find $C_2(2) = 8$, by \cite[Theorem 3.12.4(iv)]{drg}. 
\end{remark}

An amply regular graph with parameters $(v, k, a_1, c_2)$ is a $k$-regular graph with $v$ vertices such that any two adjacent vertices have exactly $a_1$ common neighbors and any two vertices at distance $2$ have $c_2$ common neighbors. We call an amply regular graph with diameter $2$ strongly regular. Neumaier \cite{-m} proved the following theorem which is called the $\mu$-bound for strongly regular graphs.

\begin{theorem}\label{srg}{\rm \cite[Theorem 3.1]{-m}} Let $G$ be a coconnected strongly regular graph with parameters $(v, k, a_1, c_2)$ and integral smallest eigenvalue $-\lambda\leq -2$. Then
$$c_2 \leq \lambda^3 (2\lambda -3).$$

\end{theorem}
The condition $-\lambda\leq -2$ implies that $G$ is not a union of cliques of the same size. Since the only strongly regular graphs which are not coconnected, are the  complete multipartite graphs, we obtain the following theorem.

\begin{theorem}

Let $G$ be an amply regular graph with parameters $(v, k, a_1, c_2)$. Let $\lambda \geq 2$ be an integer. Then there exists a real number $C_3(\lambda)$ such that if $G$ has smallest eigenvalue at least $-\lambda$, then $c_2 \leq C_3(\lambda)$ or $G$ is a complete multipartite graph.

\end{theorem}

\begin{proof}

Let $C_3(\lambda) = \max\{M(\lambda)-1,\lambda^3 (2\lambda -3)\}$. If $c_2 > C_3(\lambda)$, then $G$ has diameter $2$ by Proposition \ref{tool}. By Theorem \ref{srg}, $G$ is not coconnected. Hence, $G$ is a complete multipartite graph.

\end{proof}

\end{document}